\documentclass[11pt,oneside]{amsart}
\usepackage{hyperref}
\usepackage[a4paper,margin=3cm]{geometry}
\usepackage{amssymb}
\usepackage[T1]{fontenc}
\usepackage[utf8]{inputenc}
\usepackage[color=yellow]{todonotes}
\usepackage{enumerate}

\usepackage{thmtools, thm-restate}
\declaretheorem[name=Theorem]{thm}
\declaretheorem[name=Proposition,sibling=thm]{prop}
\declaretheorem[name=Lemma,sibling=thm]{lem}
\declaretheorem[name=Definition,sibling=thm]{defn}
\declaretheorem[name=Corollary,sibling=thm]{cor}

\newcommand{\conv}{\operatorname{conv}}
\newcommand{\xc}{\operatorname{xc}}

\newcommand{\R}{\mathbb{R}}
\newcommand{\Z}{\mathbb{Z}}
\newcommand{\AND}{\wedge}
\newcommand{\OR}{\vee}
\newcommand{\NOT}{\neg}
\newcommand{\pitch}{p}
\newcommand{\notch}{\nu}
\newcommand{\supp}{\operatorname{supp}}

\title{Strengthening Convex Relaxations of 0/1-Sets Using Boolean Formulas}

\author{Samuel Fiorini \and Tony Huynh \and Stefan Weltge}

\begin{document}

\begin{abstract}
    In convex integer programming, various procedures have been developed to strengthen convex relaxations of sets of integer points.
    On the one hand, there exist several general-purpose methods that strengthen relaxations without specific knowledge of the set $ S $ of feasible integer points, such as popular linear programming or semi-definite programming hierarchies.
    On the other hand, various methods have been designed for obtaining strengthened relaxations for very specific sets $ S $ that arise in combinatorial optimization.

    We propose a new efficient method that interpolates between these two approaches.
    Our procedure strengthens any convex set containing a set $ S \subseteq \{0,1\}^n $ by exploiting certain additional information about $ S $.
    Namely, the required extra information will be in the form of a Boolean formula defining the target set $ S $.

    The aim of this work is to analyze various aspects regarding the strength of our procedure.
    As one result, interpreting an iterated application of our procedure as a hierarchy, our findings simplify, improve, and extend previous results by Bienstock and Zuckerberg on covering problems.
\end{abstract}

\maketitle

\section{Introduction}
In convex integer programming, there exist various procedures to strengthen convex relaxations of sets of integer points.
Formally, given a set $ S \subseteq \Z^n $ of integer points and a convex set $ Q \subseteq \R^n $ with $ Q \cap \Z^n = S $, these methods aim to construct a new convex set $ f(Q) $ satisfying $ S \subseteq f(Q) \subseteq Q $.
While the main goal is to find a set $ f(Q) $ that is a better approximation of the convex hull of $ S $ than $ Q $, such procedures are usually required to be efficient in the following sense: given an (linear or semidefinite) extended formulation for $ Q $, there should be an efficient way to construct an extended formulation for $ f(Q) $.

On the one hand, there exist several general-purpose methods that strengthen relaxations \emph{without specific knowledge} of the set $ S $, in a systematic way.
The hierarchies of Sherali \& Adams~\cite{SA1990}, Lov\'asz \& Schrijver~\cite{LS1991}, and Lasserre~\cite{Lasserre2001}, which are tailored to $ 0/1 $-sets $ S \subseteq \{0,1\}^n $, are methods of that type.
Another well-studied procedure is to map the relaxation $ Q $ to its Chv\'atal-Gomory closure $ f(Q) $.
However, in general the latter is not efficient in the above sense because for polyhedra $ Q $ with a small number of facets it may produce relaxations $ f(Q) $ not admitting any extended formulation of small size.
(This is to be expected since determining membership in $ f(Q) $ is $\mathsf{NP}$-complete~\cite{Eisenbrand99}.)
A striking example is given by choosing $ Q $ as the fractional matching polytope \cite[Section 30.2]{SchrijverBookA03}. In this case, the Chv\'atal-Gomory closure $ f(Q) $ is the matching polytope. It was recently shown by Rothvo\ss{} that all extended formulations of the matching polytope have exponential size~\cite{Rothvoss17JACM}.

On the other hand, various methods have been designed for obtaining strengthened relaxations for \emph{specific} sets $ S $.
Such methods include, as an example, an impressive collection of families of valid inequalities of the traveling salesperson polytope that strengthen the classical subtour elimination formulation. 
Similar research has been performed for many other polytopes arising in combinatorial optimization, such as stable set polytopes and knapsack polytopes.

In this work, we propose a new method that interpolates between the two approaches described above.
We design a procedure to strengthen any convex set $ Q \subseteq \R^n $ containing a set $ S \subseteq \{0,1\}^n $ by exploiting certain additional information about $ S $.
Namely, the required extra information will be in the form of a Boolean formula $ \phi $ defining the target set $ S $.
The improved relaxation is obtained by ``feeding'' $ Q $ into the formula $ \phi $, and will be denoted by $ \phi(Q) $.
While the formula $ \phi $ has to be provided as a further input, we present a simple and efficient procedure to obtain $ \phi(Q) $ that is easy to implement.

To exemplify our approach, let us suppose that the set $ S $ arises from a $ 0/1 $-covering problem, i.e., it is given by a matrix $ A \in \{0,1\}^{m \times n} $ such that $ S = \{ x \in \{0,1\}^n : Ax \geqslant e \} $, where $ e $ is an all-ones vector.
Equivalently, $ S $ can be specified by the Boolean formula in conjunctive normal form
\begin{equation} \label{eq:canonical_formula}
    \phi := \bigwedge_{i = 1}^m \bigvee_{j : A_{ij} = 1} x_j,
\end{equation}
where the entries $ x_j $ of a vector $ x \in \{0,1\}^n $ are treated as the logical values `true' (if $ x_j = 1 $) or `false' (if $ x_j = 0 $).
Starting with any convex relaxation $ Q \subseteq [0,1]^n $ that contains $ S $, our procedure will yield a stronger convex relaxation $ \phi(Q) $ sandwiched between $ S $ and $ Q $. 
As we will later discuss in detail, our results simplify, improve, and extend results by Bienstock \& Zuckerberg~\cite{BZ2004} who developed a powerful hierarchy tailored to $ 0/1 $-covering problems.

The aim of this work is to analyze various aspects regarding the strength of our procedure.
An important property of $ \phi(Q) $ is that it can be described by an extended formulation whose size is bounded by the size of the formula $ \phi $ (which we formally define later) times the size of an extended formulation defining the input relaxation $ Q $.
Since many sets $ S \subseteq \{0,1\}^n $ that arise in combinatorial optimization can be defined by formulas whose size is polynomial in $ n $, for such sets the complexity of $ \phi(Q) $ can be polynomially bounded in terms of the complexity of $ Q $.

Another property of our procedure is that it is \emph{complete} in the sense that iterating it a finite number of times (in fact, at most $ n $ times) always yields the convex hull of $ S $.
Furthermore, our procedure can be applied to \emph{any} convex set $ Q \subseteq [0,1]^n $ that contains the target set $ S $.
In particular, the set $ Q $ is even allowed to contain $0/1$-points that do not belong to $ S $.
As an example, we can always apply our method with $ Q = [0,1]^n $.
Intuitively, this is possible since the information of which points belong to $ S $ is stored in $ \phi $.

Our paper is organized as follows.
We start by describing our procedure to obtain $ \phi(Q) $ in Section~\ref{secDescription}.
In Section~\ref{secPitchNotch}, we introduce notions that allow us to quantify the strength of our relaxations.
Our main results regarding properties of the set $ \phi(Q) $ are presented in Section~\ref{secMainResults}.
In Section~\ref{secApplications}, we discuss several applications of our method in detail.

For readers familiar with circuit complexity, we mention that our work is inspired by a relatively unknown connection between \emph{Karchmer-Wigderson games}~\cite{KW90} and nonnegative factorizations, pioneered by Hrube\v s~\cite{Hrubes12}.
This connection was recently rediscovered by G{\"o\"o}s, Jain and Watson~\cite{GJW2016} and exploited in~\cite{BFHS2017}.

\section{Description of the procedure}
\label{secDescription}
In order to present the construction of the procedure, let us fix some notation concerning Boolean formulas.
We consider formulas that are built out of input variables $ x_1,\dotsc,x_n $, conjunctions $ \AND $, disjunctions $ \OR $, and negations $ \NOT $ in the standard way.
Here, we define the \emph{size}\footnote{Some sources also count the number of occurrences of $ \AND $, $ \OR $ and $ \NOT $, which is not necessary for our purposes.} of a Boolean formula as the total number of occurrences of input variables. We denote by $|\phi|$ the size of $\phi$.

Given a Boolean formula $ \phi $, we can interpret it as a function from $ \{0,1\}^n \to \{0,1\} $ and for an input $ x = (x_1,\dotsc,x_n) \in \{0,1\}^n $ we will denote its output by $ \phi(x) $.
We say that the set $ S = \{ x \in \{0,1\}^n : \phi(x) = 1 \} $ is \emph{defined} by $ \phi $.
Two formulas are said to be equivalent if they define the same set.

We say that a formula is \emph{reduced} if negations are only applied to input variables.
Note that, by De Morgan's laws, every Boolean formula can be brought into an equivalent reduced formula of the same size.
As an example, the formulas
\begin{align}
    \nonumber
    \phi_1 & = \NOT \left( \left( x_1 \AND \NOT x_2 \right) \OR \left( \NOT \left( x_1 \OR x_3 \right) \right) \right) \\
    \label{eqExampleReducedFormula}
    \phi_2 & = \left( \NOT x_1 \OR x_2 \right) \AND \left( x_1 \OR x_3 \right)
\end{align}
are equivalent and both have size $ 4 $, but only the second is in reduced form.

Below, we will repeatedly use the elementary fact that for every reduced formula $\phi$ of size $|\phi|$, one of the following holds:
\begin{itemize}
    \item $|\phi| = 1$ and either $\phi = x_i$ or $\phi = \NOT x_i$ for some $i \in [n]$, or
    \item $|\phi| \geqslant 2$ and $\phi$ is either the conjuction or the disjunction of two reduced formulas $\phi_1,\phi_2$ such that $ |\phi| = |\phi_1| + |\phi_2| $.
\end{itemize}
This gives a way to represent any reduced Boolean formula as a rooted tree each of whose inner nodes is labeled with $\AND$ or $\OR$ and each of whose leaves is labeled with a non-negated variable $x_i$ or a negated variable $\NOT x_i$.  Note that there may be many trees that represent the same reduced Boolean formula, but this will not matter. Observe that the size of a formula is the number of leaves in any one of its trees.

We are ready to describe our method to strengthen a convex relaxation of a given set of points in $ \{0,1\}^n $.
\begin{defn}
    Let $ \phi $ be a reduced Boolean formula with input variables $ x_1,\dotsc,x_n $ and let $ Q \subseteq [0,1]^n $ be any convex set.
    The set $ \phi(Q) \subseteq \R^n $ is recursively constructed from the formula $ \phi $ as follows.
    \begin{itemize}
        \item Replace any non-negated input variable $ x_i $ by the set $ \{ x \in Q : x_i = 1 \} $.
        \item Replace any negated input variable $ \NOT x_i $ by the set $ \{ x \in Q : x_i = 0 \} $.
        \item Replace any conjuction $ \AND $ of two sets by their intersection.
        \item Replace any disjunction $ \OR $ of two sets by the convex hull of their union.
    \end{itemize}
\end{defn}
As an example, given any convex set $ Q \subseteq [0,1]^3 $ and the formula $ \phi_2 $ defined
in~\eqref{eqExampleReducedFormula}, we have
\begin{align*}
    \phi_2(Q) = & \conv \big( \{ x \in Q : x_1 = 0 \} \cup \{ x \in Q : x_2 = 1 \} \big) \\
        & \cap \conv \big( \{ x \in Q : x_1 = 1 \} \cup \{ x \in Q : x_3 = 1 \} \big).
\end{align*}
In the remainder of this work, we will analyze several properties of $ \phi(Q) $.
One simple observation will be that, if $ S $ is defined by $ \phi $ and $ S \subseteq Q $, then $ S \subseteq \phi(Q) \subseteq Q $.
Furthermore, $ \phi(Q) $ is strictly contained in $ Q $ unless $ \conv(S) = Q $.
In order to quantify this improvement over $ Q $, we will introduce useful measures in the next section.

\section{Measuring the strength: Pitch and notch}
\label{secPitchNotch}
We now introduce two quantities that measure the strength of our procedure.
To this end, note that for every linear inequality in variables $ x_1, \dotsc, x_n $ we can partition $ [n] $ into sets $ I^+, I^- \subseteq [n] $ (with $ I^+ \cup I^- = [n] $ and $ I^+ \cap I^- = \emptyset $) such that the inequality can be written as
\begin{equation}
    \label{eqInequalityNormalForm}
    \sum_{i \in I^+} c_i x_i + \sum_{i \in I^-} c_i (1 - x_i) \geqslant \delta,
\end{equation}
where $ c = (c_1,\dotsc,c_n)^\intercal \in \R^n_{\ge 0} $ and $ \delta \in \R $.
Since we will only consider the intersection of $ [0,1]^n $ with the set of points satisfying such an inequality, we are only interested in inequalities where $ \delta \geqslant 0 $.
In this case, we call~\eqref{eqInequalityNormalForm} an inequality in \emph{standard form}.

The \emph{notch} of an inequality in standard form is the smallest number $ \notch $ such that
\begin{equation}
    \label{eqSubsetOfCoefficientsLargeEnough}
    \sum_{j \in J} c_j \geqslant \delta
\end{equation}
holds for every $ J \subseteq [n] $ with $ |J| \geqslant \notch $, while its \emph{pitch} is the smallest number $ \pitch $ such that~\eqref{eqSubsetOfCoefficientsLargeEnough} holds for every $ J \subseteq \supp(c) $ with $ |J| \geqslant \pitch $.
Note that the pitch of an inequality is at most its notch.
For instance, the notch of the inequality $x_1 + x_n \geqslant 1$ is~$n-1$, while its pitch equals~$1$.
Both quantities appear in the study of Chv\'atal-Gomory closures of polytopes in $ [0,1]^n $.

Intuitively, the notch of an inequality is related to how ``deep'' it cuts the $0/1$-cube. For simplicity, assume that $I^- = \emptyset$, so that the origin minimizes the left-hand side of \eqref{eqInequalityNormalForm} over the cube. The notch of \eqref{eqInequalityNormalForm} is then the smallest number $ \notch $ such that no $0/1$-vector of Hamming weight $ \notch $ or more is cut by the inequality. A similar intuition applies to the pitch.

We extend the definition of notch from inequalities to sets of $0/1$-points as follows. The \emph{notch} of a non-empty set $ S \subseteq \{0,1\}^n $, denoted $\notch(S)$, is the largest notch of any inequality in standard form that is valid for $ S $. It can be shown that $ \notch(S) $ is equal to the smallest number $ k $ such that every $ k $-dimensional face of $ [0,1]^n $ contains a point from $ S $.
This equivalent definition of notch\footnote{To avoid possible confusion, we warn the reader that in a previous version of~\cite{BFHW2016}, this notion is called pitch instead of notch.} was introduced in~\cite{BFHW2016}. The main result of~\cite{BFHW2016} is that if $S$ has bounded notch and $\conv(S)$ has bounded facet coefficients, then every polytope $Q \subseteq [0,1]^n$ whose set of $0/1$-points is $S$ has bounded Chv\'atal-Gomory rank. 

The term pitch was used by Bienstock \& Zuckerberg~\cite{BZ2004}, who defined it for \emph{monotone} inequalities in standard form, that is, where $ I^- = \emptyset $.  Bounded pitch inequalities are related to the Chv\'atal-Gomory closure as follows. Consider any constants $\varepsilon > 0$ and $\ell \in \Z_{\geqslant 1}$, and any relaxation $Q := \{x \in [0,1]^n : Ax \geqslant b\}$ of a set $S := Q \cap \{0,1\}^n$, with $A$, $b$ nonnegative. Bienstock \& Zuckerberg~\cite[Lemma 2.1]{BZ2006} proved that adding all valid pitch-$\pitch$ inequalities for $\pitch \leqslant \lceil \ell / \ln(1+\varepsilon) \rceil = \Theta(\ell / \varepsilon)$ to the system defining $Q$ gives a relaxation $R$ that is a $(1+\varepsilon)$-approximation\footnote{Here, this means that $\min \{c^\intercal x : x \in f^\ell(Q)\} \leqslant (1+\varepsilon) \min \{c^\intercal x : x \in R\}$ for every nonnegative cost vector $c$, where $f^\ell(Q)$ denotes the $\ell$-th Chv\'atal-Gomory closure of $Q$.} of the $\ell$-th Chv\'atal-Gomory closure of $Q$.

\section{Main results}
\label{secMainResults}
In this section, we prove several properties of the set $ \phi(Q) $.
Let us start with the following simple observation.
\begin{prop}
    \label{propRelax}
    For every reduced Boolean formula $\phi$ and every convex set $Q \subseteq [0,1]^n$, the set $\phi(Q)$ is a convex subset of $Q$. Moreover, $\phi(Q)$ contains every point $x \in \{0,1\}^n$ such that $x \in Q$ and $\phi(x) = 1$. In other words, $\phi(Q)$ contains $Q \cap \phi^{-1}(1)$.
\end{prop}
\begin{proof}
    The fact that $\phi(Q)$ is a convex set contained in $Q$ is clear, since $\phi(Q)$ is constructed from faces of $Q$ by taking intersections and convex hulls of unions. 
    
    We prove the second part by induction on the size of $\phi$. If $|\phi| = 1$, then $\phi$ is either $\phi = x_i$ or $\phi = \NOT x_i$ for some $i \in [n]$. So either $\phi(Q) = \{x \in Q : x_i = 1\}$ or $\phi(Q) = \{x \in Q : x_i = 0\}$, respectively. We see immediately that $\phi(Q)$ contains $Q \cap \phi^{-1}(1)$. 
    
    Now if $|\phi| \geqslant 2$, then $\phi$ is the conjunction or disjunction of two formulas of smaller size, say $\phi_1$ and $\phi_2$. In the first case, $\phi = \phi_1 \AND \phi_2$ and we have
    $\phi(Q) = \phi_1(Q) \cap \phi_2(Q) \supseteq (Q \cap \phi^{-1}_1(1)) \cap (Q \cap \phi^{-1}_2(1)) = Q \cap (\phi^{-1}_1(1) \cap  \phi^{-1}_2(1)) = Q \cap \phi^{-1}(1)$, where the inclusion follows from induction. In the second case, $\phi = \phi_1 \OR \phi_2$ and $\phi(Q) = \conv(\phi_1(Q) \cup \phi_2(Q)) \supseteq (Q \cap \phi^{-1}_1(1)) \cup (Q \cap \phi^{-1}_2(1)) = Q \cap (\phi^{-1}_1(1) \cup \phi^{-1}_2(1)) = Q \cap \phi^{-1}(1)$.
\end{proof}
Next, we argue that we can use $\phi$ to transform any extended formulation for $ Q $ into one for $ \phi(Q) $.
Recall that an \emph{extended formulation of size $ m $} of a polytope $ P $ is determined by matrices $ T \in \R^{n \times d} $, $ A \in \R^{m \times d} $ and vectors $ t \in \R^n $, $ b \in \R^m $ such that $ P = \{ x \in \R^n : \exists y \in \R^d : Ay \geqslant b, \, x = Ty + t\} $.
The \emph{extension complexity} of $ P $ is defined as the smallest size of any extended formulation for $ P $, and is denoted by $ \xc(P) $.
In what follows, we need the following standard facts about extension complexity.
First, if $ F $ is a non-empty face of $ P $, then $ \xc(F) \le \xc(P) $.
Second, for any non-empty polytopes $ P_1, P_2 \subseteq \R^n $ one has $ \xc(P_1 \cap P_2) \le \xc(P_1) + \xc(P_2) $.
Third, a slight refinement of Balas' theorem~\cite{Balas79} states that $ \xc(\conv(P_1 \cup P_2)) \le \max \{\xc(P_1), 1\} + \max \{\xc(P_2), 1\} $, see~\cite[Prop.~3.1.1]{W2016}.
\begin{prop} \label{prop:xc_upper_bound}
    Let $ \phi $ be a reduced Boolean formula and let $ Q \subseteq [0,1]^n $ be a polytope such that $ \phi(Q) \ne \emptyset $.
    Then $ \phi(Q) $ is a polytope with extension complexity $ \xc(\phi(Q)) \le |\phi| \xc(Q) $.
\end{prop}
\begin{proof}
    First, note that if $ \xc(Q) = 0 $, then $ Q $ is a single point and so is $ \phi(Q) $, which implies $ \xc(\phi(Q)) = 0 $ and hence the claimed inequality holds trivially.
    Thus, we may assume that $ \xc(Q) \geqslant 1 $ holds.

    We prove the claim by induction over the size of $ \phi $.
    If $ |\phi| = 1 $, then $ \phi = x_i$ or $\phi = \NOT x_i$ for some $i \in [n]$.
    So either $\phi(Q) = \{x \in Q : x_i = 1\}$ or $\phi(Q) = \{x \in Q : x_i = 0\}$, respectively.
    In both cases, $ \phi(Q) $ is a face of $ Q $ and hence $ \xc(\phi(Q)) \leqslant \xc(Q) $.
    
    If $ |\phi| \geqslant 2 $, there exist reduced Boolean formulas $ \phi_1,\phi_2 $ (of size smaller than $ |\phi| $) with $ |\phi| = |\phi_1| + |\phi_2| $ such that $ \phi = \phi_1 \AND \phi_2 $ or $ \phi = \phi_1 \OR \phi_2 $.
    First, consider the case $ \phi = \phi_1 \AND \phi_2 $, in which we have $ \phi(Q) = \phi_1(Q) \cap \phi_2(Q) $.
    Since $ \phi(Q) $ is non-empty, the same holds for $ \phi_1(Q) $ and $ \phi_2(Q) $ and hence, by the induction hypothesis, we have $ \xc(\phi_i(Q)) \leqslant |\phi_i| \xc(Q) $ for $ i = 1,2 $.
    Therefore,
    \[
        \xc(\phi(Q)) \leqslant \xc(\phi_1(Q)) + \xc(\phi_2(Q))
        \leqslant |\phi_1| \xc(Q) + |\phi_2| \xc(Q) = |\phi| \xc(Q).
    \]
    It remains to consider the case $ \phi = \phi_1 \OR \phi_2 $, in which we have $ \phi(Q) = \conv(\phi_1(Q) \cup \phi_2(Q)) $.
    Note that the claimed inequality holds if $ \phi_1(Q) = \emptyset $ or $ \phi_2(Q) = \emptyset $.
    Thus, we may assume that $ \phi_1(Q) $ and $ \phi_2(Q) $ are both non-empty. By the induction hypothesis, $ \xc(\phi_i(Q)) \leqslant |\phi_i| \xc(Q) $ for $ i = 1,2 $.
    Therefore,
    \begin{align*}
        \xc(\phi(Q)) & \leqslant \max \{\xc(\phi_1(Q)), 1\} + \max \{\xc(\phi_2(Q)), 1\} \\
        & \leqslant \max \{|\phi_1| \xc(Q), 1\} + \max \{|\phi_2| \xc(Q), 1\} \\
        & = |\phi_1| \xc(Q) + |\phi_2| \xc(Q) \\
        & = |\phi| \xc(Q). \qedhere
    \end{align*}
\end{proof}

We remark that the upper bound provided by Proposition~\ref{prop:xc_upper_bound} is quite generous, and can be improved in some cases. For instance, if we let $\tau$ denote the number of maximal rooted subtrees of $\phi$ whose nodes are either input variables or $\AND$ gates, then we have $\xc(\phi(Q)) \leqslant \tau \xc(Q)$. This is due to the well-known fact that any intersection of faces of $Q$ is a face of $Q$.

A Boolean formula is \emph{monotone} if it does not contain negations. We are ready to prove our main theorem in the monotone case.

\begin{thm} \label{thmMainPitch}
    Let $ \phi $ be a monotone Boolean formula defining a set $ S \subseteq \{0,1\}^n $ and let $ Q \subseteq [0,1]^n $ be any convex set containing $ S $.
    If $ Q $ satisfies all monotone inequalities of pitch at most $ \pitch $ that are valid for $ S $, then $ \phi(Q) $ satisfies all monotone inequalities of pitch at most $ \pitch + 1 $ that are valid for $ S $.
    Moreover, if $ Q $ is a polytope defined by an extended formulation of size $ \sigma $, then $ \phi(Q) $ is a polytope that can be defined by an extended formulation of size $ |\phi| \sigma$, where $ |\phi| $ is the size of the formula.
\end{thm}
\begin{proof}
The second part of the theorem is implied by Proposition~\ref{prop:xc_upper_bound}. For the first part, consider any monotone pitch-$(\pitch+1)$ inequality in standard form that is valid for $S = \{x \in \{0,1\}^n : \phi(x) = 1\}$,
\begin{equation}
    \label{eqMonotoneInequality}
    \sum_{i \in I^+} c_i x_i \geqslant \delta.
\end{equation}
By the definition of pitch, we may assume $c_i > 0$ for all $i \in I^+$.  We also assume $\delta>0$; otherwise, there is nothing to prove.  Let $a \in \{0,1\}^n$ be the characteristic vector of $[n] \setminus I^+$. Thus, $a_i = 1$ if $i \in [n] \setminus I^+$ and $a_i = 0$ if $i \in I^+$. Notice that $a$ violates \eqref{eqMonotoneInequality}. This implies $\phi(a) = 0$.

By contradiction, suppose that \eqref{eqMonotoneInequality} is not valid for $\phi(Q)$. That is, there exists a point in $\phi(Q)$ that violates \eqref{eqMonotoneInequality}.
Let $T$ be a tree that represents the formula $\phi$.  Each $v \in V(T)$ has a corresponding formula, which is the formula computed by the subtree of $T$ rooted at $v$.  For notational convenience, we identity each node of $T$ with its corresponding formula.

Our strategy is to find a root-to-leaf path in $T$ such that for every node $\psi$ on this path, 
\begin{center}
($\star$) \quad $\psi(a) = 0$ \quad and \quad there exists a point $\tilde{x} = \tilde{x}(\psi) \in \psi(Q)$ that violates \eqref{eqMonotoneInequality}.
\end{center}
This is satisfied at the root node $\phi$. 

Now consider any non-leaf node $\psi$ in $T$ that satisfies ($\star$). Let $\psi_1$ and $\psi_2$ denote the children of $\psi$, so that $\psi = \psi_1 \AND \psi_2$ or $\psi = \psi_1 \OR \psi_2$. We claim that, in both cases, there exists an index $k \in \{1,2\}$ such that $\psi_k$ satisfies ($\star$). 

First, in case $\psi = \psi_1 \AND \psi_2$, we let $\tilde{x}(\psi_1) = \tilde{x}(\psi_2) := \tilde{x}(\psi)$ and choose $k \in \{1,2\}$ such that $\psi_k(a) = 0$. Such an index is guaranteed to exist since $\psi(a) = 0$. Then $\psi_k$ satisfies ($\star$).

Second, in case $\psi = \psi_1 \OR \psi_2$, we have $\psi_1(a) = \psi_2(a) = \psi(a) = 0$. We let $\tilde{x}(\psi_1)$ and $\tilde{x}(\psi_2)$ be any points of $\psi_1(Q)$ and $\psi_2(Q)$ (respectively) such that the segment $[\tilde{x}(\psi_1),\tilde{x}(\psi_2)]$ contains $\tilde{x}$. For at least one $k \in \{1,2\}$, the point $\tilde{x}(\psi_k)$ violates \eqref{eqMonotoneInequality}. Thus $\psi_k$ satisfies ($\star$) for that choice of $k$.

By iterating the argument above, starting at the root node $\phi$, we reach a leaf node $\psi$ that satisfies ($\star$). Note that  $\psi = x_j$ for some $j$, since $\phi$ is monotone. We have $a_j = \psi(a) = 0$, so $j \in I^+$. Moreover, there exists a point $\tilde{x} = \tilde{x}(\psi) \in \psi(Q) = \{x \in Q : x_j = 1\}$ that violates \eqref{eqMonotoneInequality}.

Now consider the monotone inequality
\begin{equation}
    \label{eqInequalitySmallerPitch}
    \sum_{\substack{i \in I^+ \\ i \neq j}} c_i x_i \geqslant \delta - c_j\,.
\end{equation}
This inequality is valid for $S$ since it is the sum of \eqref{eqMonotoneInequality} and $c_j (1-x_j) \geqslant 0$, which are both valid.  Since $c_j (1-\tilde{x}_j) = 0$, \eqref{eqInequalitySmallerPitch} is also violated by $\tilde{x} \in \psi(Q) \subseteq Q$.  The key observation is that the pitch of \eqref{eqInequalitySmallerPitch} is at most $\pitch$, which contradicts our assumption that $Q$ satisfies all monotone inequalities of pitch at most $\pitch$.
\end{proof}

In the non-monotone case, we now prove a statement analogous to Theorem~\ref{thmMainPitch} where the pitch is replaced by the notch.

\begin{thm}
    \label{thmMainNotch}
    Let $ \phi $ be a reduced Boolean formula defining a set $ S \subseteq \{0,1\}^n $ and let $ Q \subseteq [0,1]^n $ be any convex set containing $ S $.
    If $ Q $ satisfies all inequalities of notch at most $ \notch $ that are valid for $ S $, then $ \phi(Q) $ satisfies all inequalities of notch at most $ \notch + 1$ that are valid for $ S $.
    Moreover, if $ Q $ is a polytope defined by an extended formulation of size $ \sigma $, then $ \phi(Q) $ is a polytope that can be defined by an extended formulation of size $ |\phi| \sigma $.
\end{thm}
\begin{proof}
The proof is almost identical to that of Theorem~\ref{thmMainPitch}. Instead of repeating the whole proof, here we only explain the differences. The starting point is a notch-$(\notch+1)$ inequality 
\begin{equation}
    \label{eqStandardFormBis}
    \sum_{i \in I^+} c_i x_i + \sum_{i \in I^-} c_i (1 - x_i) \geqslant \delta\,,
\end{equation}
where $I^+ \subseteq [n] $ and $I^- \subseteq [n] $ form a partition of $ [n] $ (i.e., $ I^+ \cap I^- = \emptyset $ and $ I^+ \cup I^- = [n] $), $\delta >0$,  and $c_i \geqslant 0$ for all $i \in [n]$.  Contrary to the previous proof, here we allow $c_i=0$.  Let $a \in \{0,1\}^n$ be the characteristic vector of $I^-$.  Notice that $a$ violates \eqref{eqStandardFormBis}. This implies $\phi(a) = 0$. 

Let $T$ be a tree that represents the formula $\phi$. 
Using the same proof strategy, we find a leaf node $\psi = x_j$ or $\psi = \NOT x_j$ of $T$ such that $\psi(a) = 0$, and there exists a point $\tilde{x} = \tilde{x}(\psi) \in \psi(Q)$ that violates \eqref{eqStandardFormBis}.

If $\psi = x_j$, then $j \in I^+$ and we consider the valid inequality 
\begin{equation*} 
    \sum_{\substack{i \in I^+ \\ i \neq j}} c_i x_i + \sum_{i \in I^-} c_i (1 - x_i) + \delta (1-x_j) \geqslant \delta - c_j\,.
\end{equation*}
Otherwise, $\psi = \NOT x_j$ and thus $j \in I^-$. In this case, we consider the valid inequality
\begin{equation*}
    \sum_{i \in I^+} c_i x_i + \sum_{\substack{i \in I^- \\ i \neq j}} c_i (1 - x_i) + \delta x_j \geqslant \delta - c_j\,.
\end{equation*}
Since \eqref{eqStandardFormBis} is a notch-$(\notch+1)$ inequality, it is easy to check that the notch of both of the above inequalities is at most $\nu$. However, they are violated by the point $\tilde{x} = \tilde{x}(\psi) \in Q$. As in the proof of Theorem~\ref{thmMainPitch}, this gives the desired contradiction.
\end{proof}

Setting $ \phi^1(Q) := \phi(Q) $ and $ \phi^{\ell + 1}(Q) := \phi(\phi^\ell(Q)) $ for $ \ell \in \Z_{\ge 1} $, and using the trivial fact that the notch of a non-trivial inequality is at most $ n $, we immediately obtain the following corollary.
\begin{cor} \label{cor:complete}
    Let $ \phi $ be a reduced Boolean formula defining a set $ S \subseteq \{0,1\}^n $ and let $ Q \subseteq [0,1]^n $ be any convex set containing $ S $.
    Then we have $ \phi^n(Q) = \conv(S) $. 
\end{cor}

Another consequence of Theorem~\ref{thmMainNotch} is that integer points not belonging to $ S $ are already excluded from $ \phi(Q) $.
\begin{cor}
    Let $ \phi $ be a reduced Boolean formula defining a set $ S \subseteq \{0,1\}^n $ and let $ Q \subseteq [0,1]^n $ be any convex set containing $ S $.
    Then we have $ \phi(Q) \cap \Z^n = S $.
\end{cor}
\begin{proof}
    It suffices to show that no point from $ \{0,1\}^n \setminus S $ is contained in $ \phi(Q) $.
    To this end, fix $ \bar x \in \{0,1\}^n \setminus S $ and consider the inequality
    \[
        \sum_{i \in [n] : \bar{x}_i = 0} x_i + \sum_{i \in [n] : \bar{x}_i = 1} (1 - x_i) \geqslant 1\,,
    \]
    which is violated by $ \bar x $, but valid for all other points of $\{0,1\}^n$.
    Since the inequality has notch $ 1 $, by Theorem~\ref{thmMainNotch} it is also valid for $ \phi(Q) $ and hence $ \bar x $ is not contained in $ \phi(Q) $.
\end{proof}

\section{Applications}
\label{secApplications}
In this section, we present several applications of our procedure, in which we repeatedly make use of Theorems~\ref{thmMainPitch} and~\ref{thmMainNotch}. 
\subsection{Monotone formulas for matching}
As a first application, we demonstrate how our findings together with Rothvo\ss' result~\cite{Rothvoss17JACM} on the extension complexity of the matching polytope yield a very simple proof of a seminal result of Raz \& Wigderson~\cite[Theorem 4.1]{RW92}, which states\footnote{The original result of Raz \& Wigderson states that the depth of any monotone circuit computing the mentioned function is $ \Omega(n) $, which is equivalent to the mentioned result, see, e.g., \cite{wegener1983relating}.} that any \emph{monotone} Boolean formula deciding whether a graph on $ n $ nodes contains a perfect matching has size $ 2^{\Omega(n)} $. Before giving any further detail, we point out that Raz \& Wigderson's result extends to the bipartite case~\cite[Theorem 4.2]{RW92}, which is not the case of the polyhedral approach described below.

The fact that Rothvo\ss' theorem implies  Raz \& Wigderson's was first discovered by G{\"o\"o}s, Jain and Watson~\cite{GJW2016}.
While their arguments are based on connections between nonnegative ranks of certain slack matrices and Karchmer-Wigderson games, which implicitly play an important role in the proofs of Theorems~\ref{thmMainPitch} and \ref{thmMainNotch}, our results yield a straightforward proof that does not require any further notions.

To this end, let $ n \in \Z_{\ge 2} $ be even and let $ G = (V, E) $ denote the complete undirected graph on $ n $ nodes.
The set $ S $ considered by Raz \& Wigderson is the set
\[
    S := \{ x \in \{0,1\}^E : \supp(x) \subseteq E \text{ contains a perfect matching} \}.
\]
Let $ \phi $ be any monotone Boolean formula in variables $ x_e $ ($ e \in E $) that defines $ S $.
Next, define the polytope
\[
    \label{eqOddCutInequalities}
    P := \{ x \in [0,1]^E : x(\delta(U)) \geqslant 1 \text{ for every } U \subseteq V \text{ with } |U| \text{ odd} \}.
\]
It is a basic fact that $ S $ is contained in $ P $.
Furthermore, observe that every non-trivial inequality in the definition of $ P $ has pitch $ 1 $.
Thus, we have $ \conv(S) \subseteq \phi([0,1]^E) \subseteq P $.
Moreover, if we consider the affine subspace
\[
    D := \{ x \in \R^E : x(\delta(\{u\}) = 1 \text{ for every } u \in V \},
\]
it is well-known that both $ \conv(S) \cap D $ and $ P \cap D $ are equal to the perfect matching polytope of $ G $, and hence we obtain that  $ \phi([0,1]^E) \cap D $ is also equal to the perfect matching polytope of $ G $.
By Rothvo\ss' result, this implies $ \xc(\phi([0,1]^E)) = 2^{\Omega(n)} $.
On the other hand, by Proposition~\ref{prop:xc_upper_bound} we also have $ \xc(\phi([0,1]^E)) \le |\phi| \cdot \xc([0,1]^E) = |\phi| \cdot  2 |E| \le n^2 |\phi| $ and hence $ |\phi| $ must be exponential in $ n $.

\subsection{Covering problems: the binary case}
In this section, we consider sets $ S \subseteq \{0,1\}^n $ that arise from $ 0/1 $-covering problems, in which there is a matrix $ A \in \{0,1\}^{m \times n} $ such that $ S = \{ x \in \{0,1\}^n : Ax \geqslant e \} $, where $ e $ is an all-ones vector.
As an example, if $ A $ is the node-edge incidence matrix of an undirected graph $ G $, then the points of $ S $ correspond to vertex covers in $ G $.
This shows that, in general, the convex hull of such sets $ S $ may not admit polynomial-size (in $ n $) extended formulations, see, e.g.,~\cite{bazzi2015no,FMPTW15,GJW2016}.

Moreover, general $ 0/1 $-hierarchies may have difficulties identifying basic inequalities even in simple instances.
For example, in~\cite{BZ2006} it is shown that if $ Ax \geqslant e $ consists of the inequalities $ \sum_{i \in [n] \setminus \{j\}} x_i \geqslant 1 $ for each $ j \in [n] $, then it takes at least $ n - 2 $ rounds of the Lov\'asz-Schrijver or Sherali-Adams hierarchy to satisfy the pitch-2 inequality $ \sum_{i \in [n]} x_i \geqslant 2 $.

By developing a hierarchy tailored to $ 0/1 $-covering problems, Bienstock \& Zuckerberg~\cite{BZ2004} were able to bypass some of these issues.
As their main result, for each $ k \in \Z_{\ge 1} $, they construct a polytope $ f^k(Q) $ containing $ S $ satisfying the following two properties.
First, every inequality of pitch at most $ k $ that is valid for $ S $ is also valid for $ f^k(Q) $.
Second, $ f^k(Q) $ can be described by an extended formulation of size $ (m + n)^{g(k)} $, where $ g(k) = \Omega(k^2) $.

However, constructing the polytope $ f^k(Q) $ is quite technical and involved.
Prior to our work, it had been simplified by Mastrolilli~\cite{Mastrolilli17} by using a modification of the Sherali-Adams hierarchy that is based on appropriately defined high-degree polynomials.
However, the extended formulations in~\cite{Mastrolilli17} are not smaller than those of Bienstock \& Zuckerberg~\cite{BZ2004}.

In contrast, our procedure directly implies significantly simpler and smaller extended formulations that satisfy all pitch-$ k $ inequalities.
Indeed, consider the canonical monotone formula \eqref{eq:canonical_formula}.
Theorem~\ref{thmMainPitch} gives a polyhedral relaxation $ \phi^k([0,1]^n) $ of $ S $ whose points satisfy all valid inequalities of pitch at most $ k $, and that can be defined by an extended formulation of size at most $2n \cdot (mn)^k$.

\subsection{Covering problems: bounded coefficients}
Next, we consider a more general form of a covering problem in which $ S = \{ x \in \{0,1\}^n : Ax \geqslant b \} $ for some non-negative integer matrix $ A \in \Z_{\ge 0}^{m \times n} $ and $ b \in \Z_{\ge 1}^m $.
We first restrict ourselves to the case that all entries in $ A $ and $ b $ are bounded by some constant $ \Delta \in \Z_{\ge 2} $.

Based on their results in~\cite{BZ2004}, Bienstock \& Zuckerberg~\cite{BZ2006} provide an extended formulation of size $ O(m + n^\Delta)^{g(k)} $, where $ g(k) = \Omega(k^2) $.  Our method yields a significantly smaller extended formulation, via the following lemma.

\begin{lem} \label{lemMonotoneBounded}
For every $ A \in \Z_{\ge 0}^{m \times n} $ and $ b \in \Z_{\ge 1}^m $ with entries bounded by $ \Delta $, the set $ S = \{ x \in \{0,1\}^n : Ax \geqslant b \} $ can be defined by a monotone formula $ \phi $ of size at most $ \Delta^{5.3} m n \log^{O(1)}(n) $. Moreover, this formula can be constructed in randomized polynomial time.
\end{lem}
\begin{proof}
Fix $ i \in [m] $, let $ n' := \sum_{j=1}^n A_{ij} $ and let $ \psi_i $ be a monotone formula defining the set $ \{ y \in \{0,1\}^{n'} : \sum_{k=1}^{n'} y_k \ge b_i\} $. Next, pick any function $ h : [n'] \to [n] $ such that $ |h^{-1}(j)| = A_{ij} $ for all $ j \in [n] $. In formula $\psi_i$, replace every occurrence of $ y_{k} $ by $ x_{h(k)} $, for $ k \in [n'] $. We obtain a monotone formula $\phi_i$ defining the set $ \{x \in \{0,1\}^n : \sum_{j=1}^n A_{ij} x_j \ge b_i\} $. By using the construction of Hoory, Magen and Pitassi~\cite{HMP06} for the initial formula $ \psi_i $, the resulting formula $ \phi_i $ has size
$$
|\phi_i| = |\psi_i| \leq \Delta^{4.3} n' \log^{O(1)}(n'/\Delta) \leq \Delta^{5.3} n \log^{O(1)} (n)\,,
$$
since $ n' \leq n \Delta $ and $ b_i \leq \Delta $.

The result follows by taking $ \phi := \bigwedge_{i=1}^m \phi_i $.
\end{proof}

Letting $\phi$ be the formula from Lemma~\ref{lemMonotoneBounded} and applying Theorem~\ref{thmMainPitch}, the relaxation
$ \phi^k([0,1]^n) $ of $ S $ satisfies all valid pitch at most $k$ inequalities for $S$, and  $ \phi^k([0,1]^n) $ can be defined by an extended formulation of size at most $2n \cdot (\Delta^{5.3} m n \log^{O(1)}(n))^k$.

\subsection{Covering problems: the general case}
In some cases, especially when $ m = O(1) $, the matrix $ A \in \Z_{\ge 0}^{m \times n} $ and vector $ b \in \Z_{\ge 0}^m $ may have coefficients as large as $ 2^{\Omega(n \log n)} $.
For such general instances, we can improve the bound of the previous paragraph.
Beimel \& Weinreb~\cite{beimel2006monotone} show that, for every $ a_1,\dotsc,a_n,\delta \in \R_{\ge 0} $, the set $ \{ x \in \{0,1\}^n : \sum_{j=1}^n a_j x_j \geqslant \delta \} $ can be decided by a monotone formula of size $ n^{O(\log n)} $.
Thus, defining $ \phi $ as a conjunction of such formulas for each inequality in $ Ax \geqslant b $, we obtain a monotone formula deciding $ S $ of size $ m n^{O(\log n)} $.
Thus, again by Theorem~\ref{thmMainPitch}, the polytope $ \phi^k([0,1]^n) $ provides a relaxation for $ S $ that satisfies all valid pitch at most $ k $ inequalities for $S$, and that can be defined by an extended formulation of size at most $ \left( m n^{O(\log n)} \right)^k $.
In comparison, for this general case, Bienstock \& Zuckerberg~\cite{BZ2004} have no nontrivial upper bound.

\subsection{Constant notch $0/1$-sets}

In this section, we consider non-empty sets $S \subseteq \{0,1\}^n$ with constant notch $ \notch(S) $.
These sets have several desirable properties.
For example, as noted in~\cite{BFHS2017} (and implicitly in~\cite{CL2016}), there is an easy polynomial-time algorithm to optimize a linear function over a constant notch set $S$, provided that we have a polynomial-time membership oracle for $S$.
On the other hand, sets with constant notch do not necessarily admit small extended formulations.  
Indeed, counting arguments developed in~\cite{Rothvoss13MPA,AKW2016} show that even for a ``generic'' set $ S \subseteq \{0,1\}^n $ with notch $ \notch(S) = 1 $, $\conv(S)$ requires extended formulations of size $ 2^{\Omega(n)} $.

This raises the question of which constant notch sets \emph{do} admit compact extended formulations.  Theorem~\ref{thmMainNotch} provides a nice partial answer: if $S$ has constant notch and a small formula description, then $\conv(S)$ has small extension complexity. In particular, if $S$ has constant notch and $S$ can be described by a formula $\phi$ of polynomial-size (in $n$), then $\conv(S)$ can be described by a polynomial-size extended formulation. Notice that every explicit $0/1$-set $S$ of constant notch such that $\xc(\conv(S))$ is large would thus provide an explicit Boolean function requiring large depth circuits, and solve one of the hardest open problems in circuit complexity.

\section{Concluding remarks}
In this paper, we propose a new method for strengthening convex relaxations of $0/1$-sets. Given a Boolean formula defining the $0/1$-set and a relaxation contained in the $0/1$-cube, we obtain a new relaxation by ``feeding'' the relaxation in the formula. This provides the current simplest and smallest linear extended formulations expressing inequalities of constant pitch in the monotone case, and constant notch in the general case.

For covering problems, we obtain extended formulations whose size is often vastly smaller than those of Bienstock \& Zuckerberg~\cite{BZ2004,BZ2006}. This is possible since we allow \emph{any} monotone formula, and can thus use any known construction from the literature. In contrast, Bienstock \& Zuckerberg implicitly only consider formulas in conjunctive normal form. The number of clauses in every formula in conjunctive normal form is at least the number of minimal covers\footnote{A point $x \in \{0,1\}^n$ is called a \emph{min-term}, and its characteristic vector a \emph{minimal cover}, if $x$ is a minimal element of $S$ with respect to the component-wise order $\leqslant$.}, which makes it impossible for them to construct small extended formulations in situations where the number of minimal covers is large.

To conclude, we state a few open questions raised by our work.

\begin{enumerate}
    \item Do the new extended formulations lead to any new interesting algorithmic application, in particular for covering problems? This appears to be connected to the following question. How good are the lower bounds obtained after performing a few rounds of the Chv\'atal-Gomory closure? For some problems, such as the vertex cover problem in graphs or more generally in $q$-uniform hypergraphs with $q = O(1)$, the bounds turn out to be quite poor in the worst case~\cite{ST10,bazzi2015no}. The situation is less clear for other problems, such as network design problems. Recent work~\cite{FGKS17} on the tree augmentation problem uses certain inequalities from the first Chv\'atal-Gomory closure in an essential way. For the related $2$-edge connected spanning subgraph problem, our work implies that one can approximately optimize over the $\ell$-th Chv\'atal-Gomory closure in quasi-polynomial time, for $\ell = O(1)$.
    
    \item For which classes of polytopes in $ [0,1]^n $ can one approximate a constant number of rounds of the Chv\'atal-Gomory closure with compact extended formulations? Recent results of Mastrolilli~\cite{Mastrolilli17} show that this is possible for packing problems, via a compact \emph{positive semi-definite} extended formulation.
    
    \item Can one find polynomial-size monotone formulas for any nonnegative weighted threshold function, that is, for every min-knapsack $\{x \in \{0,1\}^n : \sum_{i=1}^n a_i x_i \geqslant \beta\}$? This would improve on the $n^{O(\log n)}$ upper bound by Beimel and Weinreb~\cite{beimel2006monotone}. Klabjan, Nemhauser and Tovey show that separating pitch-$1$ inequalities for such sets is NP-hard~\cite{KNT98}. However, this does not rule out a polynomial-size extended formulation defining a relaxation that would be stronger than that provided by pitch-$1$ inequalities.
\end{enumerate}




\section{Acknowledgments}
This work was done while Samuel Fiorini was visiting the Simons Institute for the Theory of Computing. It was supported in part by the DIMACS/Simons Collaboration on Bridging Continuous and Discrete Optimization through NSF grant \#CCF-1740425, and in part by ERC Consolidator Grant 615640-ForEFront. 

\bibliography{references}{}
\bibliographystyle{amsplain}

\end{document}